\documentclass[11pt,a4paper]{article}
\usepackage{preamble}
\usepackage{graphicx}
\usepackage{cleveref}

\usepackage{authblk}
\usepackage{blindtext}
\tikzset{edgee/.style = {->,> = latex'}}
\newcolumntype{P}[1]{>{\centering\arraybackslash}p{#1}}
\newcolumntype{M}[1]{>{\centering\arraybackslash}m{#1}}

\title{\textbf{Combinatorial Interpretations of $q$-Fibonacci Numbers and Their Binomial Analogues}}
\author{Nived J M}

\date{}

\affil{\footnotesize M.Sc Mathematics, 
Indian Institute of Technology Hyderabad, India\\
Email: nivedjm.res@gmail.com
}

\begin{document}

\maketitle

\begin{abstract}
The Fibonomial coefficients are well-known analogues of the classical binomial coefficients. In 2009, Sagan and Savage introduced a combinatorial interpretation for these coefficients, based on tiling a rectangular grid. More recently, Bergeron extended this work by providing a similar interpretation for the $q$-Fibonomial coefficients, using weighted tilings of a rectangular grid. Inspired by Bennett's model, Bergeron also developed a staircase tiling model for the $q$-Fibonomial coefficients. While Bergeron's proofs for the rectangular grid model relied on induction, and the staircase model on bijective correspondences with the rectangular grid model, these approaches lacked deeper structural insights.
In this paper, we propose a novel model for the $q$-Fibonacci numbers that generalizes Bergeron's approach. This new model not only enables us to prove several identities related to $q$-Fibonacci numbers but also provides a non-bijective proof for the staircase model of the $q$-Fibonomial coefficients, offering greater structural clarity. Additionally, we demonstrate new identities involving the $q$-Fibonomial coefficients using this refined rectangular grid model, further enhancing the combinatorial understanding of these mathematical objects.
\end{abstract}


\hspace{-.9 cm} \textbf{Keywords :} $q-$ Fibonomial coefficients, Partition Tiling of a grid, Sagan and Savage model, Bennett's model, Bergeron's model\\
\textbf{2020 Mathematics Subject Classification :} 05B45, 05A10, 05A19, 05B05

\section{Introduction}

The Fibonacci sequence $0, 1, 1, 2, 3, 5, 8, \dots$ is among the most captivating subjects in mathematics. Each term is the sum of the two preceding terms, formally defined by the recurrence relation $F_n = F_{n-1} + F_{n-2}$. Fibonomial coefficients are defined analogously to binomial coefficients for $n, k \in \mathbb{N}$ and $n \geq k \geq 0$:
\[
\begin{bmatrix}
n\\
k
\end{bmatrix}_{F} = \frac{F_n F_{n-1} \dots F_{n-k+1}}{F_1 F_2 \dots F_k}.
\]
We define the $q$-analog of the $n$-th Fibonacci number as $[F_n]_q = \sum_{i=0}^{F_n-1} q^i$. Similarly, for $n, k \in \mathbb{N}$ and $n \geq k \geq 0$, the $q$-analog of the Fibonomial coefficient is given by:
\[
\begin{bmatrix}
n\\
k
\end{bmatrix}_{\mathcal{F}_{\scriptscriptstyle q}} := \frac{[F_{m+n}]_q!}{[F_m]_q! [F_n]_q!}.
\]
In 2009, Sagan and Savage \cite{2} provided a combinatorial interpretation for Lucasnomials, which generalize these fibonomial coefficients. Their interpretation involves tiling an $m \times n$ grid with unit squares (monominos) and $1 \times 2$ rectangular tiles (dominos). Specifically, for $m,n \geq 1$, the quantity $\begin{bmatrix}
m+n\\
n
\end{bmatrix}_F$ counts the number of ways to draw a lattice path from $(0,0)$ to $(m,n)$, then tile each row above the lattice path with monominos and dominos, and tile each column below the lattice path with monominos and dominos, with the restriction that the column tilings cannot start with a monomino. In 2011, Elizabeth Reiland \cite{1} presented combinatorial proofs for certain identities of Fibonomial coefficients  using this tiling model. Later, Arthur T. Benjamin and Elizabeth Reiland \cite{beli} gave further insights by proving few more identities. In 2020, Bennett \cite{3} introduced a new combinatorial interpretation using a staircase structure. The Fibonomial number $\begin{bmatrix}
n\\
k
\end{bmatrix}_F$ counts the number of specific partial tilings of a Young diagram with a staircase shape of size $n$. This model is a direct bijective version of the Sagan and Savage model. However, the structural proof provided for this model in Bennett's paper is particularly notable, as the Sagan and Savage model was presented using an inductive proof.

In 2020, Bergeron \cite{4} proposed two analogous interpretations for $q$-Fibonomial coefficients, specifically using a rectangular grid model and a staircase model. In the same paper, a combinatorial interpretation for $q$-Fibonacci numbers is presented by tiling a $1 \times n$ rectangular strip with monominos and dominos, assigning specific weights to each tile. These weighted tilings are utilized in the inductive proofs of the aforementioned interpretations.

We propose the barrier tiling model, a new interpretation of $q$-Fibonacci numbers that generalizes the existing combinatorial approach \cite{4}. \Cref{sbar} details the core principles of this method and its applications. This interpretation allows us to derive a structural proof for Bergeron's model of $q$-Fibonomial numbers, offering deeper insight into its framework. The structural proof, along with further demonstrations of identities related to $q$-Fibonomial coefficients, will be presented in \Cref{sber}.

\section{A modified interpretation for $q-$Fibonacci numbers}\label{sbar}

The number of distinct ways to tile an $n \times 1$ rectangular strip using monomino and domino tiles is $F_{n+1}$ \cite{proof}. In a similar way, Bergeron, Ceballos, and Kustner \cite{4} introduced a combinatorial framework for $q$-Fibonacci numbers by assigning specific weights to the tiles, referred to as the $q$-weight. In their approach, monominos are given a weight of 1, while a domino covering the $(i-1)$-th and $i$-th cells of the strip is assigned a weight of $q^{F_i}$. The weight of a tiling $T$, denoted by $\omega(T)$, is the product of the weights of its individual tiles. The total weight across all such tilings of a $1\times n$ strip corresponds to $[F_{n+1}]_q$. In this section, we will introduce a more generalized model for $q$-Fibonacci numbers.

\begin{lemma}\cite{4}\label{lemsup}
For $a,b \in \mathbb{N}$ :
\begin{equation*}
    [a.b]_{q} = [a]_{q}[b]_{q^a}
\end{equation*}
\begin{equation*}
    [a+b]_{q} = [a]_{q}+q^a[b]_{q}
\end{equation*}
\end{lemma}

\begin{lemma}(cf.~\cite{4})\label{iden}
For $n,k \in \mathbb{N}$ and $k\leq n$:
\begin{equation}
    F_{n+1}=F_{n-k+1}F_{k+1}+F_{n-k}F_{k}\label{nop}
\end{equation}
\begin{equation}\label{main}
    [F_{n+1}]_{q} = [F_{k+1}]_{q^{F_{n-k+1}}}[F_{n-k+1}]_{q} +  q^{F_{n-k+1}F_{k+1}}[F_{n-k}]_{q^{F_{k}}}[F_{k}]_{q} 
\end{equation}
\end{lemma}
\begin{proof}
 The identity \ref{nop} can be proven using the double counting method. The left-hand side, $F_{n+1}$, counts the number of ways to tile a rectangular strip of length $n$. By placing a bar after the $(n-k)$-th cell, the right-hand side accounts for two distinct cases: tilings where no domino crosses the bar and those where a domino does. The second identity, \ref{main}, follows directly from identity \ref{nop} in conjunction with Lemma \ref{lemsup}.
\end{proof}
\begin{figure}[h!]
\centering
\begin{subfigure}[b]{1\textwidth}
 \centering
    \begin{tikzpicture}
\draw [black, thick] (0,1) to (0,2) to (10,2) to (10,1) to (0,1)  ;

\draw [black, thick] (1,1) to (1,2) ;
\draw [black, thick] (2,1) to (2,2) ;
\draw [black, thick] (3,1) to (3,2) ;
\draw [blue, line width=2.5pt] (4,1) to (4,2) ;
 \draw [black, thick] (5,1) to (5,2) ;
\draw [black, thick] (6,1) to (6,2) ;
\draw [black, thick] (7,1) to (7,2) ;
\draw [black, thick] (8,1) to (8,2) ;
\draw [black, thick] (9,1) to (9,2) ;

\filldraw[blue] (4,1) circle (4pt) ;

\filldraw[black] (6.5,1.5) circle (2pt) ;
\filldraw[black] (7.5,1.5) circle (2pt) ;
\draw [black,ultra thick] (7.5,1.5) to (6.5,1.5) ;

\filldraw[black] (1.5,1.5) circle (2pt) ;
\filldraw[black] (2.5,1.5) circle (2pt) ;
\draw [black,ultra thick] (1.5,1.5) to (2.5,1.5) ;

\draw[<->] (0,2.5) -- (3,2.5);
\draw [black, thick] (1.5,2.5) node[anchor=south] {floor=3};

\draw[<->] (10,2.5) -- (                                                            6,2.5);
\draw [black, thick] (8,2.5) node[anchor=south] {floor=4};

\draw [black, thick] (2,0.1) node[anchor=south] {$d_{1}$};
\draw [black, thick] (8,0.1) node[anchor=south] {$d_{2}$};

\end{tikzpicture}
\caption{$I$ barrier with \textit{height} = 6+1 = 7.}
   \end{subfigure}
\vspace{0.0005 cm}

\begin{subfigure}[b]{1\textwidth}
 \centering
    \begin{tikzpicture}
\draw [black, thick] (0,1) to (0,2) to (10,2) to (10,1) to (4,1);
\draw [black, thick] (0,1) to (3,1);

\draw [black, thick] (1,1) to (1,2) ;
\draw [black, thick] (2,1) to (2,2) ;
\draw [blue, line width=2.5pt] (3,1) to (3,2) ;
\draw [black, thick] (4,1) to (4,2) ;
 \draw [black, thick] (5,1) to (5,2) ;
\draw [black, thick] (6,1) to (6,2) ;
\draw [black, thick] (7,1) to (7,2) ;
\draw [black, thick] (8,1) to (8,2) ;
\draw [black, thick] (9,1) to (9,2) ;
\draw [blue, line width=2.5pt] (3,1) to (4,1) ;

\filldraw[blue] (4,1) circle (4pt) ;

\filldraw[black] (8.5,1.5) circle (2pt) ;
\filldraw[black] (7.5,1.5) circle (2pt) ;
\draw [black,ultra thick] (7.5,1.5) to (8.5,1.5) ;

\filldraw[black] (1.5,1.5) circle (2pt) ;
\filldraw[black] (2.5,1.5) circle (2pt) ;
\draw [black, ultra thick] (1.5,1.5) to (2.5,1.5) ;

\filldraw[red] (3.5,1.5) circle (2pt) ;
\filldraw[red] (4.5,1.5) circle (2pt) ;
\draw [red, ultra thick] (3.5,1.5) to (4.5,1.5) ;

\draw [black, thick] (2,0.1) node[anchor=south] {$d_{3}$};
\draw [black, thick] (8,0.1) node[anchor=south] {$d_{4}$};
\draw [black, thick] (4,0.1) node[anchor=south] {$d_{5}$};


\draw[<->] (3,2.5) -- (10,2.5);
\draw [black, thick] (6.5,2.5) node[anchor=south] {floor = 7};

\end{tikzpicture}
\caption{$L$ barrier with \textit{height} = 3+1 = 4.}
   \end{subfigure}

    \caption{Demonstration of floors and heights of dominos in a $1\times 10$ strip with barrier point 4.}
   \label{fig:IL BARR}
\end{figure}
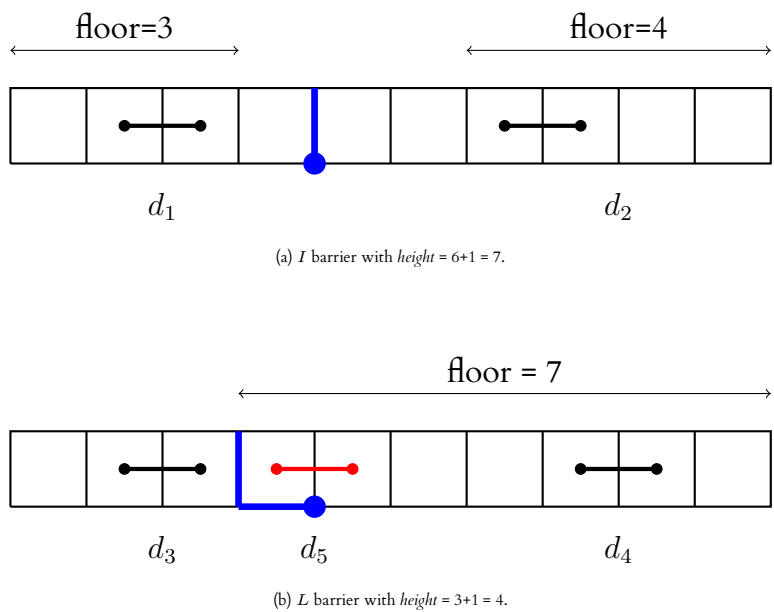

\subsection{The Barrier Tiling method}\label{bari}
 
 We begin by fixing a point along the bottom of a $1 \times n$ rectangular strip, located after $k$ cells from the left end, referred to as the \textit{barrier point} and denoted by $k$. From this barrier point, two types of barriers can be introduced: the $I$ barrier and the $L$ barrier. An $I$ barrier moves one unit vertically, while an $L$ barrier moves one unit vertically after first moving one unit to the left. These barriers partition the strip into two compartments, which are tiled independently using monominos and dominos. A constraint is imposed on tilings with an $L$ barrier, where the first tile in the right compartment must be a domino, referred to as the \textit{special domino}. Such a tiling, associated with a specific barrier point, is called a \textit{barrier tiling}. The set of all barrier tilings of a $1 \times n$ strip with barrier point $k$ is denoted by $\mathcal{B}(n,k)$.

In line with a key application of this interpretation, we adopt terminology from Bergeron's model \cite{4}. The \textit{height} $h$ of an $I$ barrier is defined as one plus the number of cells in the right compartment. For an $L$ barrier, the height is one plus the number of cells in the left compartment. The \textit{height} of a domino $d$ denoted by $h(d)$, is defined as the height of its associated barrier, while the \textit{floor} $f(d)$ of a domino represents the number of cells between the end of the strip in the same compartment and the farthest end of the domino. Using these definitions, the following weights are assigned to the tiles:
\begin{itemize}
    \item Each monomino has a weight of 1.
    \item A domino $d$ in the right compartment of an $I$ barrier or the left compartment of an $L$ barrier has a weight of $q^{F_{f(d)}}$.
    \item A domino $d$ in the left compartment of an $I$ barrier or the right compartment of an $L$ barrier, excluding the special domino, has a weight of $q^{F_{f(d)} F_{h(d)}}$.
    \item The special domino has a weight of $q^{F_{f(d)} F_{h(d)+1}}$.
\end{itemize}
 Given the assigned weights, we denote the weight of a tiling $T$ as $B(T)$ and that of a domino $d$ as $B(d)$. Here, $B(T)$ is simply the product of the weights of the dominos in the tiling, since monominos have a weight of 1.

Consider the example illustrated in Figure \ref{fig:IL BARR}. Since $d_1$ is located in the left compartment of the $I$ barrier, its weight is $B(d_1) = q^{F_{f(d_1)} F_{h(d_1)}} = q^{F_3 F_7}$, whereas $d_2$, in the right compartment, has a weight of $B(d_2) = q^{F_{f(d_2)}} = q^{F_4}$. Similarly, we calculate $B(d_3) = q^{F_{f(d_3)} F_{h(d_3)}} = q^{F_3 F_4}$ and $B(d_4) = q^{F_{f(d_4)}} = q^{F_3}$. Note, however, that the special domino $d_5$ has a weight of $B(d_5) = q^{F_{f(d_5)} F_{h(d_5)+1}} = q^{F_5 F_5}$.

\begin{theorem}\label{bthm}
The total weight of all barrier tilings of a $1 \times n$ strip, with respect to the barrier point $k$, is given by $[F_{n+1}]_{q}$.
\begin{equation*} 
  [F_{n+1}]_{q} = \sum_{T \in \mathcal{B}(n,k)} B(T)
\end{equation*}
\end{theorem}

\begin{proof}
Consider an $I$ barrier placed at the barrier point $k$. This partitions the strip into left and right compartments consisting of $k$ and $n-k$ cells, respectively. The total weight of all possible tilings in the right compartment is clearly $[F_{n-k+1}]_{q}$, as we are assigning $q$-weights to the tiles in this region (cf.~\cite{4}). Each domino in the left compartment has a weight given by $q^{F_{f(d)}F_{h(d)}} = (q^{F_{h}})^{F_{f(d)}} = (q^{F_{n-k+1}})^{F_{f(d)}}$. This can be interpreted as a $q$-weighted tiling of a $1 \times k$ strip with parameter $q^{F_{n-k+1}}$. Therefore, the total weight of all tilings in the left compartment is $[F_{k+1}]_{q^{F_{n-k+1}}}$. Hence, the total weight of all tilings with an $I$ barrier at position $k$ is $[F_{n-k+1}]_{q}[F_{k+1}]_{q^{F_{n-k+1}}}$.

Next, consider the case where an $L$ barrier is placed at the barrier point $k$. Since the special domino covers the $k$-th and $(k+1)$-th cells, its floor value is $n-k+1$. Since the height of the barrier is $k$, the weight of the special domino is $q^{F_{f(d)}F_{h(d)+1}} = q^{F_{n-k+1}F_{k+1}}$. The left compartment has $[F_{k}]_{q}$ possible tilings, each assigned a $q$-weight. Similarly, as in the $I$ barrier case, the right compartment (excluding the special domino) can be viewed as a $q$-weighted tiling of a $1 \times (n-k-1)$ strip with parameter $q^{F_{k}}$. Therefore, the total weight of all tilings with an $L$ barrier at position $k$ is $q^{F_{n-k+1}F_{k+1}}[F_{k}]_{q}[F_{n-k}]_{q^{F_{k}}}$.

By applying Lemma \ref{iden}, it follows that the sum of these weights equals $[F_{n+1}]_{q}$, completing the proof.
\end{proof}

\subsection{Identities}

Using the barrier tiling interpretation, we will present combinatorial proofs for both established and new $q$-Fibonacci identities in this segment.

\begin{proposition}\label{ip1}\cite{4}
For $n,k\in \mathbb{N}$, $n\geq 1$ and $0\leq k\leq n$:
    \begin{equation*}
    [F_{n}]_{q} = [F_{k+1}]_{q^{F_{n-k}}}[F_{n-k}]_{q} +  q^{F_{n-k}F_{k+1}}[F_{n-k-1}]_{q^{F_{k}}}[F_{k}]_{q} 
\end{equation*}
\end{proposition}

\begin{proof}
    The proposition serves as an alternative formulation of Lemma \ref{iden}.  The left-hand side signifies the total number of ways to barrier tile a \(1 \times (n-1)\) strip with the barrier point located at \(k\). The first and second expressions on the right-hand side account for the number of barrier tilings corresponding to an \(I\) barrier and an \(L\) barrier at \(k\), respectively. 
\end{proof}

\begin{corollary}\cite{sch}
    For $n\geq 2$:
    \begin{equation*}
    [F_{n}]_{q} = [F_{n-1}]_{q} +  q^{F_{n-1}}[F_{n-2}]_{q}.
\end{equation*}
\end{corollary}
\begin{proof}
    Choosing barrier point $k$ as $n-2$ in the previous proposition leads to this identity.
\end{proof}

\begin{proposition}\label{propid1}
    For $n,a\in\mathbb{N}$, $n\geq 1$ and $1\leq a\leq n+1$:
    \begin{multline*}
    [F_{n+1}]_{q} = [F_{n-a+2}]_{q^{F_{a}}}   [F_{a}]_{q}  \\
    +
    [F_{a}]_{q}\mathlarger{\mathlarger{\sum}}_{i=1}^{\floor*{\frac{n}{a}}-1} 
    \left \{  q^{F_{a}\sum_{j=1}^{i} F_{n-ja+2}} 
    \left( \prod_{j=1}^{i}[F_{a-1}]_{q^{F_{n-ja+1}}}\right)
    [F_{n-(i+1)a+2}]_{q^{F_{a}}} \right \} 
    \\
    + 
    q^{F_{a}\mathlarger{\sum}_{j=1}^{\floor*{\frac{n}{a}}} F_{n-ja+2}}  \left(\prod_{j=1}^{\floor*{\frac{n}{a}}}[F_{a-1}]_{q^{F_{n-ja+1}}}\right)
   \left [F_{n-\floor*{\frac{n}{a}}a+1}\right ]_{q}
    \end{multline*}
\end{proposition}


\begin{proof}
    The given identity can be proven combinatorially using the barrier tiling method, which offers a clearer and more accessible approach compared to algebraic techniques. We begin by fixing the barrier point at $n-(a-1)$ in a $1 \times n$ rectangular strip. There are $[F_{n+1}]_q$ possible ways to barrier tile this strip.

\textbf{Step 0:} Consider the barrier positioned at \(n-(a-1)\), is of type \(I\). The number of tilings of the strip in this case is given by $[F_{n-a+2}]_{q^{F_{a}}} [F_{a}]_q$.

\textbf{Step 1:} Consider the barrier at $n-(a-1)$ is of type $L$. The right compartment can be tiled in $q^{F_{a}F_{n-a+2}}[F_{a-1}]_{q^{F_{n-a+1}}}$ ways. The left compartment, which has $n-a$ cells, can be tiled in $[F_{n-a+1}]_q$ ways. Note that, an equivalent weight can be obtained by applying the barrier tiling method to the left compartment, rather than relying on the $q$-weighted tiling approach. So we can install a new barrier $a-1$ cells from the rightmost edge of the compartment, at position $n-2a+1$. When the new barrier is of type $I$, there are $[F_{n-2a+2}]_{q^{F_{a}}}[F_{a}]_q$ possible tilings. Therefore, the total number of tilings with an $L$ barrier at $n-a+1$ and an $I$ barrier at $n-2a+1$ is $q^{F_{a}F_{n-a+2}}[F_{a-1}]_{q^{F_{n-a+1}}}[F_{n-2a+2}]_{q^{F_{a}}}[F_{a}]_q$. 

 \textbf{Step 2:} If the new barrier at $n-2a+1$ is of type $L$, the same process is repeated by introducing another barrier point at $n-3a+1$.

\textbf{Step $i$:} It can be observed that the set of all barrier points is given by \(\{n - ja + 1 \mid 1 \leq j \leq i+1, j \in \mathbb{N}\}\). Among these points, all but the \((i+1)^{\text{th}}\) barrier are of type \(L\). The floor value of each special domino associated with the \(L\) barriers is \(a\), since we place the barriers \(a-1\) cells apart at each step. Note that the distance between consecutive barrier points is \(a\) cells. Thus, the height of the \(j^{\text{th}}\) barrier is \(n - ja + 1\) for $1\leq j\leq i$, and consequently, all special dominoes together contribute a combined weight of $q^{F_{a} \sum_{j=1}^{i} F_{n - ja + 2}}$.

Excluding the special domino, the right compartment of each \(L\) barrier contains \(a-2\) cells. The number of ways to tile this right portion of the \(j^{\text{th}}\) barrier is given by \( [F_{a-1}]_{q^{F_{n - ja + 1}}} \). Therefore, the total contribution from these right portions is $\prod_{j=1}^{i} [F_{a-1}]_{q^{F_{n - ja + 1}}}$. At the \((i+1)^{\text{th}}\) barrier point, which is of type \(I\), the remaining \(1 \times (n - ia)\) compartment can be tiled in $\left[F_{n - (i+1)a + 2}\right]_{q^{F_{a}}}[F_{a}]_q$ ways. The total weight contributed at the \(i^{\text{th}}\) step is then
\[
q^{F_{a} \sum_{j=1}^{i} F_{n - ja + 2}} \left( \prod_{j=1}^{i} [F_{a-1}]_{q^{F_{n - ja + 1}}} \right) \left[F_{n - (i+1)a + 2}\right]_{q^{F_{a}}} [F_{a}]_q.
\]
Summing this expression over all \(1 \leq i \leq \left\lfloor \frac{n}{a} \right\rfloor - 1\) gives the second term on the left-hand side of the identity.

\textbf{Final step:} We introduce an $L$ barrier at the $\left\lfloor \frac{n}{a} \right\rfloor^{th}$ barrier point, located at $n-\left\lfloor \frac{n}{a} \right\rfloor a+1$. The special domino has a weight of $q^{F_{a}F_{n-\left\lfloor \frac{n}{a} \right\rfloor a+2}}$. The left and right compartments (excluding the special domino) can be tiled in $\left[F_{n-\left\lfloor \frac{n}{a} \right\rfloor a+1}\right]_q$ and $[F_{a-1}]_{q^{F_{n-\left\lfloor \frac{n}{a} \right\rfloor a+1}}}$ ways, respectively. Hence, the total number of ways to tile this compartment with an $L$ barrier at the $\left\lfloor \frac{n}{a} \right\rfloor^{th}$ barrier point is:
\[
q^{F_{a}F_{n-\left\lfloor \frac{n}{a} \right\rfloor a+2}}\left[F_{n-\left\lfloor \frac{n}{a} \right\rfloor a+1}\right]_q \left[F_{a-1}\right]_{q^{F_{n-\left\lfloor \frac{n}{a} \right\rfloor a+1}}}
\]
Thus, the total number of tilings for the entire $1 \times n$ strip, when all $\left\lfloor \frac{n}{a} \right\rfloor$ barriers are of type $L$, is:
\[
q^{F_{a}\sum_{j=1}^{\left\lfloor \frac{n}{a} \right\rfloor} F_{n-j a+2}}\left(\prod_{j=1}^{\left\lfloor \frac{n}{a} \right\rfloor}[F_{a-1}]_{q^{F_{n-ja+1}}}\right)[F_{n-\left\lfloor \frac{n}{a} \right\rfloor a+1}]_q
\]
This completes the proof for the final term on the right-hand side. 
\end{proof}

\begin{corollary}
    For $n\geq 1$:
    \begin{equation}\label{sho}
    [F_{n+1}]_{q} = [F_{n}]_{q}    +
    \sum_{i=1}^{\floor*{\frac{n}{2}}-1} 
    \left \{  q^{\sum_{j=1}^{i} F_{n-2j+2}} 
    [F_{n-2i}]_{q} \right \} 
   + q^{\sum_{j=1}^{\floor*{\frac{n}{2}}} F_{n-2j+2}}  
    \end{equation}
\end{corollary}

\begin{proof}
    The result follows directly from Proposition \ref{propid1} by substituting \( a = 2 \). 
\end{proof}

\begin{corollary}
    For $n\geq 1$:
    \begin{equation*}
        F_{n+1}=1+\sum_{j=1}^{\floor*{\frac{n}{2}}} F_{n-2j+2}
    \end{equation*}
\end{corollary}

\begin{proof}
    Substituting \( q = 1 \) in equation \eqref{sho} simplifies the \( q \)-deformed Fibonacci numbers to standard Fibonacci numbers, yielding this classical Fibonacci identity.

Alternatively, consider the highest power of \( q \) on the left-hand side of equation \eqref{sho}, which is \( q^{F_{n+1}-1} \). Therefore, the highest power of \( q \) on the right-hand side must also be \( q^{F_{n+1}-1} \). Notice that the sum \( \sum_{j=1}^{\left\lfloor \frac{n}{2} \right\rfloor} F_{n-2j+2} \) represents the highest exponent of \( q \) in the right-hand side. Equating this with \( F_{n+1}-1 \) leads to the desired identity.
\end{proof}

\begin{corollary}
    If $n,a\geq 1$ and $a$ divides $n$,  then $[F_{a}]_{q}$ is a factor of $[F_{n}]_{q}$.
\end{corollary}

\begin{proof}
    Consider the identity in Proposition \ref{propid1}. The first two terms on the right-hand side take the form \( [F_a]_q \) multiplied by a polynomial. Now, assume \( a \) divides \( n+1 \), so \( n - \left\lfloor \frac{n}{a} \right\rfloor a = n \mod a = a-1 \). This implies that the factor \( \left[F_{n-\left\lfloor \frac{n}{a} \right\rfloor a + 1}\right]_q \) on the final term of the right-hand side simplifies to \( \left[F_a\right]_q \). Since \( [F_a]_q \) is a common factor in the right-hand side, it should be a factor of \( [F_{n+1}]_q \). Here, replacing \( n+1 \) with \( n \) proves the corollary.
\end{proof}

\begin{proposition}

     For $n\geq 1$ and $0\leq a\leq n$ :
    \begin{multline*}
    [F_{n+1}]_{q} = q^{F_{a+1}F_{n-a+1}}[F_{n-a}]_{q^{F_{a}}}   [F_{a}]_{q}  \\
    +
    [F_{a}]_{q}\mathlarger{\mathlarger{\sum}}_{i=1}^{\floor*{\frac{n}{a}}-1} 
    \left \{  q^{F_{a+1}F_{n-(i+1)a+1}} 
    \left( \prod_{j=1}^{i}[F_{a+1}]_{q^{F_{n-ja+1}}}\right)
    [F_{n-(i+1)a}]_{q^{F_{a}}} \right \} 
    \\
    + 
    \left(\prod_{j=1}^{\floor*{\frac{n}{a}}}[F_{a+1}]_{q^{F_{n-ja+1}}}\right)
   \left [F_{n-\floor*{\frac{n}{a}}a+1}\right ]_{q}
    \end{multline*}

\end{proposition}

\begin{proof}
    The approach is similar to the proof of proposition \ref{propid1}. Here we will install the first barrier point $a$ on a strip of length $n$. There are $[F_{n+1}]_{q}$ ways to barrier tile the strip, which is the left hand side. The first term in the right hand side refers to the number of tilings with an $L$ barrier at $a$. Now when there is an $I$ barrier at $a$, as there are $[F_{n-a+1}]_{q}$ ways to tile the right compartment, we will install another barrier point after $a$ cells at $2a$. In the $i^{th}$ step, we have $i+1$ barrier points at $\{ja : 1\leq j\leq i+1, j\in \mathbb{N}\}$. One can see that there are $  q^{F_{a+1}F_{n-(i+1)a+1}} 
    \left( \prod_{j=1}^{i}[F_{a+1}]_{q^{F_{n-ja+1}}}\right)
    [F_{n-(i+1)a}]_{q^{F_{a}}}[F_{a}]_{q}  $ tilings with an $I$ barrier at the first $i$ barrier points ($\{ja:1\leq j\leq i\}$ ) and an $L$ barrier at $(i+1)^{th}$ barrier point ($(i+1)a)$. Considering the first $\floor*{\frac{n}{a}}-1$ steps we obtain the second expression. The remaining case is when there are $I$ barriers in all $\floor*{\frac{n}{a}}-1$ barrier points and an $L$ barrier at $\floor*{\frac{n}{a}}a$ which gives us the third term of the right hand side. 
     \end{proof}

\section{Combinatorial interpretations of $q-$Fibonomial numbers}\label{sber}

We define the factorial analog of a $q-$Fibonacci number as $[F_{m}]_{q}^{!}=\prod_{i=1}^{m}[F_{i}]_{q}$.
For $m,n\in \mathbb{N}$, the $q-$ analog of Fibonomial number is defined by 
\[
\begin{bmatrix}
m+n\\
m
\end{bmatrix}_{\mathcal{F}_{\scriptscriptstyle q}}:=\frac{[F_{m+n}]_{q}{!}}{[F_{m}]_{q}{!} [F_{n}]_{q}{!}}.
\]

Recall the notations used in the partition tiling \ref{par} from the section \ref{moy}. Similar to the $\mathcal{L}_{\lambda}^{\prime}$ defined there, we define $\mathcal{L}_{\lambda}^{\prime\prime}$ as the set of all linear tilings of partition $\lambda$ in which, the last tile covering any part is a domino. For $\lambda\subseteq m\times n$, we define $\mathcal{L}_{\lambda}\times \mathcal{L}^{\prime\prime}_{\lambda^*}$ as the set of all partition tilings of an $m\times n$ grid such that all the parts of $\lambda$ are linearly tiled independently and all parts of $\lambda^{*}$ are tiled under the last domino restriction. These mandatory dominos we will refer as special dominos. See the figure below.

We will assign coordinate weight to the tiles by its position in the grid such that the bottom left cell has position $(1,1)$. The position of a cell in the grid is $(i,j)$ if its the $i^{th}$ cell from the left and $j^{th}$ cell from the bottom. Monominos are assigned with a weight 1 and a special domino at position $(i,j)$ has a weight $q^{F_{i+1}F_{j}}$. Any ordinary domino at $(i,j)$ has a weight $q^{F_{i}F_{j}}$. For any partition tiling $T$ of an $m\times n$ grid as mentioned above $\omega(T)$ represents the weight of the tiling $T$ which is nothing but the product of weights of individual tiles.

\begin{figure}[h!]
\centering
\begin{minipage}{.475\textwidth}
  \centering

 \begin{tikzpicture}
\draw [black, thick] (0,6) to (4,6) to (4,1) to (0,1) to (0,6) ;

\draw [black, thick] (0,5) to (4,5) ;
\draw [black, thick] (0,4) to (4,4) ;
\draw [black, thick] (0,3) to (4,3) ;
\draw [black, thick] (0,2) to (4,2) ;

\draw [black, thick] (1,1) to (1,6) ;
\draw [black, thick] (2,1) to (2,6) ;
\draw [black, thick] (3,1) to (3,6) ;

\draw [blue, ultra thick ] (0,1) to (1,1) ;
\draw [blue, ultra thick ] (1,3) to (1,1) ;
\draw [blue, ultra thick ] (2,3) to (1,3) ;
\draw [blue, ultra thick ] (2,3) to (2,4) to (3,4) ;
\draw [blue, ultra thick ] (3,4) to (3,6) to (4,6) ;

\filldraw[black] (0.5,1.5) circle (2pt) ;
\filldraw[red] (1.5,1.5) circle (2pt) ;
\filldraw[black] (2.5,1.5) circle (2pt) ;
\filldraw[black] (3.5,1.5) circle (2pt) ;

\filldraw[black] (0.5,2.5) circle (2pt) ;
\filldraw[red] (1.5,2.5) circle (2pt) ;
\filldraw[red] (2.5,2.5) circle (2pt) ;
\filldraw[black] (3.5,2.5) circle (2pt) ;

\filldraw[black] (0.5,3.5) circle (2pt) ;
\filldraw[black] (1.5,3.5) circle (2pt) ;
\filldraw[red] (2.5,3.5) circle (2pt) ;
\filldraw[black] (3.5,3.5) circle (2pt) ;

\filldraw[black] (0.5,4.5) circle (2pt) ;
\filldraw[black] (1.5,4.5) circle (2pt) ;
\filldraw[black] (2.5,4.5) circle (2pt) ;
\filldraw[red] (3.5,4.5) circle (2pt) ;

\filldraw[black] (0.5,5.5) circle (2pt) ;
\filldraw[black] (1.5,5.5) circle (2pt) ;
\filldraw[black] (2.5,5.5) circle (2pt) ;
\filldraw[red] (3.5,5.5) circle (2pt) ;

\draw [red, very thick] (1.5,1.5) to (1.5,2.5) ;

\draw [black, thick] (1.8,2.25) node[anchor=south] {$q^{2}$};
\draw [black, thick] (3.8,2.25) node[anchor=south] {$q^{3}$};
\draw [black, thick] (1.8,3.25) node[anchor=south] {$q^{2}$};

\draw [black, thick] (.5,.25) node[anchor=south] {1};
\draw [black, thick] (1.5,.25) node[anchor=south] {2};
\draw [black, thick] (2.5,.25) node[anchor=south] {3};
\draw [black, thick] (3.5,.25) node[anchor=south] {4};
\draw [black, thick] (-.5,1.5) node[anchor=west] {1};
\draw [black, thick] (-.5,2.5) node[anchor=west] {2};
\draw [black, thick] (-.5,3.5) node[anchor=west] {3};
\draw [black, thick] (-.5,4.5) node[anchor=west] {4};
\draw [black, thick] (-.5,5.5) node[anchor=west] {5};

\draw [red, very thick] (2.5,3.5) to (2.5,2.5) ;
\draw [black, very thick] (3.5,1.5) to (3.5,2.5) ;

\draw [red, very thick] (3.5,4.5) to (3.5,5.5) ;

\draw [black, very thick] (1.5,5.5) to (2.5,5.5) ;

\draw [black, very thick] (1.5,3.5) to (0.5,3.5) ;

\end{tikzpicture}
\caption{Coordinate weight assignment in the partition tiling of a $5\times 4$ grid}
  \label{fig:qfibo}
\end{minipage}%
\hfill
\begin{minipage}{.475\textwidth}
  \centering

 \begin{tikzpicture}
\draw [black, thick] (0,7.25) to (0,2.25) ;
\draw [black, thick] (1,6.25) to (1,2.25) ;
\draw [black, thick] (2,5.25) to (2,2.25) ;
\draw [black, thick] (3,4.25) to (3,2.25) ;
\draw [black, thick] (4,3.25) to (4,2.25) ;

\draw [black, thick] (0,2.25) to (5,2.25) ;
\draw [black, thick] (0,3.25) to (4,3.25) ;
\draw [black, thick] (0,4.25) to (3,4.25) ;
\draw [black, thick] (0,5.25) to (2,5.25) ;
\draw [black, thick] (0,6.25) to (1,6.25) ;

\draw [blue, ultra thick ] (2,2.25) to (1,2.25) ;
\draw [blue, ultra thick ] (1,4.25) to (1,2.25) ;
\draw [blue, ultra thick ] (1,4.25) to (0,4.25) ;
\draw [blue, ultra thick ] (0,7.25) to (0,4.25) ;

\filldraw[blue] (2,2.25) circle (2pt) ;
\filldraw[blue] (0,7.25) circle (2pt) ;

\filldraw[red] (1.5,2.75) circle (2pt) ;
\filldraw[red] (2.5,2.75) circle (2pt) ;
\filldraw[black] (3.5,2.75) circle (2pt) ;

\filldraw[black] (0.5,3.75) circle (2pt) ;

\filldraw[red] (0.5,4.75) circle (2pt) ;
\filldraw[red] (1.5,4.75) circle (2pt) ;

\draw [red, very thick] (1.5,2.75) to (2.5,2.75) ;
\draw [red, very thick] (0.5,4.75) to (1.5,4.75) ;

\end{tikzpicture}\caption{A tiling of the staircase structure of length and breadth 5 with respect too a lattice path starting from $(2,0)$. }
  \label{fig:stair}
\end{minipage}
\end{figure}

\begin{theorem}\cite{4}
    The total number of ways to partition tile an $m\times n$ grid assigning the coordinate weight is $ \begin{bmatrix}
m+n\\
n
\end{bmatrix}_{\mathcal{F}} $.

\[ \begin{bmatrix}
m+n\\
n
\end{bmatrix}_{\mathcal{F}} = \sum_{\lambda\subseteq m\times n} \quad
   \sum_{T \in \mathcal{L}_{\lambda} \times\mathcal{L^{\prime\prime}}_{\lambda^*}} \omega(T)
\]
\end{theorem}

A staircase model for $q-$Fibonomial coefficients similar to the combinatorial model for Lucasnomials presented by Bennett\cite{3} was also given in the same paper. They obtained that model from the direct bijective correspondence with the above interpretation but failed to give a non bijective proof. We will see the brief discription of that model below and a proof for the same which gives more insight.

Consider a staircase shaped grid of length and breadth $n$. The first row (column) consists of $n-1$ cells whereas the last consists 0 cells. See the figure \cref{fig:stair} for reference. Introduce a coordinate system with the bottom left corner of the grid as origin, the horizontal line as X-axis and each cell is of edge length 1 unit. Fix a point $(k,0)$ in the bottom line (exactly $k$ cells away from the left most line). Now construct a lattice path starting from this point to the top left point $(0,n)$ of the staircase structure using only $I$ and $L$ steps. For such a given lattice path, we will tile the left compartment of $I$ step and right compartment of an $L$ step using monominos and dominos. Note that the tile just right to the $L$ step should be the special domino. The left compartment of $L$ step and right compartment of $I$ are left untiled. We will use $S(n,k)$ to denote the collection of all such tilings of a staircase grid of length and breadth $n$ with respect to lattice paths starting from the point $(k,0)$\footnote{For simplicity we say the partition tiling of the grid with starting point $k$ instead of $(k,0)$.}. The weights of the tiles are assigned using height and floor\footnote{Recall the definitions from the previous section} values similar to the barrier tiling method. That is the normal dominos has a weight $q^{F_{f(d)}F_{h(d)}}$ and the special dominos has a weight $q^{F_{f(d)}F_{h(d)+1}}$. Monominos are with weight 1. Here also we will denote $\omega(T)$ as the weight of a tiling $T\in S(n,k)$. 

\begin{theorem}\cite{4}
    The total weight of such partition tiling of a staircase structure of size $n$ starting from $(k,0)$ is counted by $ \begin{bmatrix}
n\\
k
\end{bmatrix}_{\mathcal{F}}$

\[
\sum_{T\in S(n,k)}\omega(T)=\begin{bmatrix}
n\\
k
\end{bmatrix}_{\mathcal{F}}
\]
\end{theorem}

\begin{proof}
    Consider a staircase grid of dimension $n$. The bottom row consists of $n-1$ cells and the topmost consists 0 cells. We will linearly tile each of these $n-1$ rows independently and denote the set of all such tilings by $S(n)$. Corresponding to any tiling $T\in S(n)$ we can construct a unique lattice path starting from $(k,0)$ such a way that the path won't cross any dominos. If there is a domino crossing, the path will take an $L$ step to avoid that. Otherwise it will move vertically upwards taking an $I$ step. The weight of any $T\in S(n)$ is assigned according to this unique lattice path. For each row, the point where the lattice path meets that row will be fixed as the barrier point of that strip. With respect to the associated barrier point, the tiles in each row are assigned with weights as defined in the barrier tiling method \ref{bari}. The product of the weights of individual tiles give the weight of the tiling.

    The barrier point of the bottom row of length $n-1$ is $(k,0)$. There are $[F_{n}]_{q}$ ways to tile this. Depending on whether the barrier is $I$ or $L$, the barrier point of the next row will be at $(k,1)$ or $(k-1,1)$ respectively. One can observe that the total number of ways to tile these two rows is $[F_{n}]_{q}[F_{n-1}]_{q}$. Using basic induction arguments its straight-forward that the total weight of all tilings of $S(n)$ is $[F_{n}]_{q}^{!}$.

  One can see that any tiling $T_{1}\in S(n,k)$ has $k-1$ untiled left compartments of length varying from 1 to $k-1$ and $n-k-1$ of such compartments in right side of length varying from 1 to $n-k-1$. Notice that linearly tiling those untiled portion will make it an element of $S(n)$. The total number of ways to tile all these compartments using $q-$weights is $[F_{k}]_{q}^{!}[F_{n-k}]_{q}^{!}$.  Consequently each tiling $T\in S(n,k) $ can be mapped to $[F_{k}]_{q}^{!}[F_{n-k}]_{q}^{!}$ tilings in $S(n)$. Hence the total weight of the tilings of $S(n,k)$ will be $\frac{[F_{n}]_{q}^{!}}{[F_{k}]_{q}^{!}[F_{n-k}]_{q}^{!}}=\begin{bmatrix}
n\\
k
\end{bmatrix}_{\mathcal{F}}$.

\end{proof}

\subsection{Identities }
The above models of $q-$Fibonomial numbers can be used to generate proofs for a few identities. As we have already seen the cobinatorial arguments using the rectangular grid in section \ref{cry}, we will make use of the staircase model here.

\begin{proposition}
     For $n\geq 1$ and $0\leq k\leq n$;
     \[
\begin{bmatrix}
n\\
k
\end{bmatrix}_{\mathcal{F}}=  [F_{k+1}]_{q^{F_{n-k}}} 
 \begin{bmatrix}
n-1\\
k
\end{bmatrix}_{\mathcal{F}}+   q^{F_{n-k}F_{k+1}} 
  [F_{n-1-k}]_{q^{F_{k}}}  \begin{bmatrix}
n-1\\
k-1
\end{bmatrix}_{\mathcal{F}}
     \]
\end{proposition}

\begin{proof}
    The left hand side is the total number of ways to partition tile a staircase structure of dimensions $n$ with respect to the starting point $k$. The first step of the path can be either $I$ or $L$. When the first step is $I$, the barrier point of the $2^{nd}$ bottom strip is $k$. We have to tile the left compartment of this $I$ barrier and partition tile the $n-1$ staircase structure (obtained excluding the first row) with starting point $k$. As the height of the barrier is $n-k$, there are $ [F_{k+1}]_{q^{F_{n-k}}} $ ways to tile the left compartment of the first row. There are $\begin{bmatrix}
n-1\\
k
\end{bmatrix}_{\mathcal{F}}$ ways to tile the $n-1$ staircase grid formed by the remaining strips. As a result there are $[F_{k+1}]_{q^{F_{n-k}}} 
 \begin{bmatrix}
n-1\\
k
\end{bmatrix}_{\mathcal{F}}$ tilings with the first step of the corresponding lattice path is an $I$ step.

Suppose the first step be an $L$ step. The barrier point is $k$ for the first row. As the height of the barrier is $k$ and the strip is of length $n-1$, the special domino weights $ q^{F_{n-k}F_{k+1}}$ and there are $ [F_{n-1-k}]_{q^{F_{k}}} $ ways to tile the remaining part of the right compartment. Since the first step is $L$, the starting point get shifted by 1. An $n-1$ staircase with $k$ as starting point can be tiled in $\begin{bmatrix}
n-1\\
k-1
\end{bmatrix}_{\mathcal{F}}$ ways. Hence the total number of tiles with first step $L$ is $ q^{F_{n-k}F_{k+1}} 
  [F_{n-1-k}]_{q^{F_{k}}}  \begin{bmatrix}
n-1\\
k-1
\end{bmatrix}_{\mathcal{F}}$. This completes the proof.
\end{proof}

\begin{proposition}
     For $n\geq 1$ and $0\leq k\leq n-1$;
     \[
\begin{bmatrix}
n\\
k
\end{bmatrix}_{\mathcal{F}}= \mathlarger{\mathlarger{\sum}}_{i=0}^{k} \left \{ q^{F_{n-k}\sum_{j=1}^{i}F_{k-j+2}}
\left ( \prod_{j=1}^{i} [F_{n-k-1}]_{q^{F_{k-j+1}}}  \right )
  [F_{k-i+1}]_{q^{F_{n-k}}}
 \begin{bmatrix}
n-i-1\\
k-i
\end{bmatrix}_{\mathcal{F}}\right \}
     \]
\end{proposition}

\begin{proof}
  One can partition tile an $n$ staircase grid with the starting point $k$ in $\begin{bmatrix}
n\\
k
\end{bmatrix}_{\mathcal{F}}$ ways. Its easy to see that any lattice path of staircase grid have their first $I$ step in between the first row and the $(k+1)^{th}$ row. Consider the lattice paths which has its first $I$ step in the $(i+1)^{th}$ row. The first $i$ rows we have $L$ barrier with $j^{th}$ row having barrier height $k-j+1$ for all $1\leq j \leq i$. Hence the special domino and the remaining part of the right compartment of $j^{th}$ row can be tiled in respectively $q^{F_{n-k}F_{k-j+2}}$ and $[F_{n-k-1}]_{q^{F_{k-j+1}}}$ number of ways. The total weight of all special dominos in the first $i$ rows will be $q^{F_{n-k}\sum_{j=1}^{i}F_{k-j+2}}$. Similarly the number of ways to tile all the right compartments excluding the special dominos is $\prod_{j=1}^{i} [F_{n-k-1}]_{q^{F_{k-j+1}}}$. As the path has moved $i$ steps to the left, the left compartment of the first $I$ barrier will be of the length $k-i$. Note that the $(i+1)^{th}$ row is of length $n-i-1$. From this, we can conclude the left compartment of $(i+1)^{th}$ row can be tiled in $[F_{k-i+1}]_{q^{F_{n-k}}}$ ways. Now the remaining portion to tile is an $n-i-1$ staircase structure. As the lattice path has took $k-i$ steps, the lattice paths of the $n-i-1$ staircase has the starting point $k-i$. There are $ \begin{bmatrix}
n-i-1\\
k-i
\end{bmatrix}_{\mathcal{F}}$ ways to tile this. Product of all these weights gives the number of ways to tile the grid when the first $I$ step is on $(i+1)^{th}$ row. As the maximum $L$ steps possible before the first $I$ is $k$, we vary $i$ from 0 to $k$. The total sum of all those weights will give the expression in the right hand side. This proves the proposition.

\end{proof}

\begin{proposition}
     For $n\geq 1$ and $0\leq k\leq n$:
    \begin{multline*}
        \begin{bmatrix}
n\\
k
\end{bmatrix}_{\mathcal{F}}= \mathlarger{\mathlarger{\sum}}_{i=0}^{n-k-2} \left \{ q^{F_{k+1}F_{n-k-i}}  [F_{n-k-i-1}]_{q^{F_{k}}}
\left ( \prod_{j=1}^{i} [F_{k+1}]_{q^{F_{n-k-j+1}}}  \right )
 \begin{bmatrix}
n-i-1\\
k-1
\end{bmatrix}_{\mathcal{F}}\right \}
\\
+\prod_{j=1}^{n-k}[F_{k+1}]_{q^{F_{n-k-j+1}}}
\end{multline*}

\end{proposition}

\begin{proof}
The proof uses a similar approach to the previous one. Here we will enumerate the number of tilings possible of an $n$ staircase grid with the starting point $k$ by counting the tilings with lattice path having its first $L$ on the $(i+1)^{th}$ row. The first $i$ rows have an $I$ barrier and for any $1\leq j\leq i$ the height of the $j^{th}$ barrier is $n-j-k+1$. As the left compartment of each of these rows has length $k$, All these left compartments can be tiled together in $\prod_{j=1}^{i} [F_{k+1}]_{q^{F_{n-k-j+1}}}$ ways. Notice that the $(i+1)^{th}$ row has length $n-i$ and the height of the corresponding $L$ barrier is $k$. From this one can see, there are $q^{F_{k+1}F_{n-k-i}}  [F_{n-k-i-1}]_{q^{F_{k}}}$ ways to tile the right compartment of that barrier. To tile the remaining $n-i-1$ staircase with the starting point $k-1$, there are $\begin{bmatrix}
n-i-1\\
k-1
\end{bmatrix}_{\mathcal{F}}$ possibilities. The product of all these weights gives the total weight of all tilings with its first $L$ step at $(i+1)^{th}$ row. Note that we can vary $i$ from 0 to $n-k$. For $i>n-k$ the path will go outside the grid. But there are no tilings with their first $L$ step at $(n-k)^{th}$ row as it will leads to just 1 cell in the right compartment of that barrier which is not possible due to the domino restriction. This happens when $i=n-k-1$. We have to sum up the weights corresponding to all $0\leq i\leq n-k-2$. This supply the first summation in the right hand side. When $i=n-k$ we have to tile the left compartment of $I$ barrier in the first $n-k$ rows and after that lattice path is forced to move such that all the remaining steps are $L$ with empty right compartment. This give rise to the second expression in the right hand side. This completes the proof as the right hand side expression calculate the partition tilings of $n$ staircase grid with lattice starting point $k$.

\end{proof}

\section{Acknowledgments}

I would like to thank my adviser Anurag Singh for guiding me through the right track.

\bibliographystyle{abbrv}
\bibliography{references}

\end{document}